\documentclass{amsart}
\usepackage[left=3.5cm, right=3.5cm, top=4cm, bottom=4cm]{geometry}

\usepackage{amsmath}
\usepackage{amsfonts}
\usepackage{amssymb}
\usepackage{amsthm}

\usepackage{enumerate}
\usepackage{graphicx}
\usepackage[hidelinks]{hyperref}

\usepackage{booktabs}

\usepackage{tikz}

\newtheorem{theorem}{Theorem}[section]

\newtheorem{proposition}[theorem]{Proposition}
\newtheorem{lemma}[theorem]{Lemma}

\theoremstyle{definition}
\newtheorem{remark}[theorem]{Remark}

\newcommand{\mcg}{\mathcal{M}}
\newcommand{\artin}{\mathcal{A}}
\newcommand{\cox}{\mathcal{C}}
\newcommand{\cgroup}{\mathcal{G}}
\DeclareMathOperator{\specli}{SL}

\title{Curves intersecting in a circuit pattern}
\author{Levi Ryffel}

\begin{document}
\begin{abstract}
  We show that the cycle relation between
  Dehn twists about curves in a circuit
  detects whether the circuit bounds an embedded disc.
  This is done by determining the isomorphism type
  of the group generated by said Dehn twists
  for various surfaces.
\end{abstract}

\maketitle

\section{Introduction}
A very well-known fact in the theory of mapping class groups
is that relations between Dehn twists about pairs of curves
detect whether the curves intersect zero, one,
or at least two times~\cite[Section~3.5]{primer}.
More precisely, let $\alpha_1$ and $\alpha_2$ be curves in a surface $S$,
and let $T_i$ be the Dehn twist about $\alpha_i$.
Then $\alpha_1$ and $\alpha_2$ are disjoint (up to homotopy) if and only if
the associated Dehn twists satisfy the
\emph{commutation relation} $T_1 T_2 = T_2 T_1$.
Similarly, $\alpha_1$ and $\alpha_2$ intersect precisely once
(again, up to homotopy) if and only if
the associated Dehn twists satisfy the
\emph{braid relation} $T_1 T_2 T_1 = T_2 T_1 T_2$.
In all other cases, $T_1$ and $T_2$ generate
a free subgroup of the mapping class group of $S$.
This article will be on another instance
of the observation
that the presence or absence of certain relations
between Dehn twists
has consequences about the constellation
of the involved curves.

Recent work on so-called bouquets of curves
(families $\alpha_1, \dots, \alpha_n$ of curves
intersecting, up to isotopy, in one common point)
shows that pairwise braid relations
and the so-called \emph{cycle relation}
$T_i T_j T_k T_i = T_j T_k T_i T_j$
or
$T_i T_k T_j T_i = T_k T_j T_i T_k$
between all triples
$\alpha_i, \alpha_j, \alpha_k$
of distinct curves
is equivalent to the family $\alpha_1, \dots, \alpha_n$
of pairwise non-isotopic curves
forming a bouquet~\cite[Theorem~1]{bouquets}.
Here, we seek to characterize circuits
of curves in a very similar way,
featuring longer cycle relations.

A family $\alpha_1, \dots, \alpha_n$ of $n$ curves
in a surface $S$
is said to form a \emph{circuit}
if each curve $\alpha_i$ intersects $\alpha_{i-1}$ and $\alpha_{i+1}$
precisely once, and is disjoint from all other curves $\alpha_j$,
where indices are taken modulo $n$.
A circuit $\alpha_1, \dots, \alpha_n$
is said to \emph{bound a disc}
if the complement $S \setminus (\alpha_1 \cup \cdots \cup \alpha_n)$
has a connected component homeomorphic to a disc.
We say that a circuit $\alpha_1, \dots, \alpha_n$
\emph{bounds an embedded closed disc} in $S$
if there is an embedding of a closed disc
such that its boundary
gets mapped to points in the
union $\alpha_1 \cup \cdots \cup \alpha_n$.

We write $\mcg(S)$ for the mapping class group
(consisting of isotopy classes of
orientation-preserving homeomorphisms
fixing the boundary)
of an orientable surface $S$.
Circuits of curves have previously been studied
by Labruère, who
showed that if $\alpha_1, \dots \alpha_n$
is a circuit bounding an embedded closed disc $\Delta$
such that when travelling in the counter-clockwise
manner around $\Delta$
the curves appear in the order $\alpha_1, \dots, \alpha_n$,
then the cycle relation
$T_n \cdots T_1 T_n \cdots T_3 = T_{n-1} \cdots T_1 T_n \cdots T_2$
holds~\cite[Proposition~2]{labruere}.
In fact, her result is considerably stronger,
see Proposition~\ref{prop:dn} below.
It is worth emphasizing the fact
that $\Delta$ being embedded is important:
if the boundary of $\Delta$
has self-intersections,
then the cycle relation may not hold.

Another result that will feature in this article
is by Mortada, asserting that
if $S$ is a certain neighbourhood
(denoted below by $M^n$, see Figure~\ref{fig:nbhds}) of a circuit
$\alpha_1, \dots, \alpha_n$,
then the homomorphism $\artin(\widetilde{A}_{n-1}) \to \mcg(S)$,
mapping the standard generator $s_i$ to $T_i$
for all $i$,
is injective~\cite[Theorem~5.5.4]{mortada}.
Here, $\artin(\widetilde A_{n-1})$ is an Artin group,
see Section~\ref{sec:artin}.
We will reprove this result and generalize it
(see Theorem~\ref{thm:regular-nbhd}
and Proposition~\ref{prop:antilde}
below) in order to prove our main result.

\begin{theorem}\label{thm:cycle-relation}
  Let $\alpha_1, \dots, \alpha_n$ be a circuit of $n \geq 3$ curves
  in a surface $S$.
  Then the circuit $\alpha_1, \dots, \alpha_n$ bounds an embedded
  closed disc $\Delta$
  if and only if one of the two cycle relations
  \[
    T_n \cdots T_1 T_n \cdots T_3 = T_{n-1} \cdots T_1 T_n \cdots T_2
  \]
  or
  \[
    T_1 \cdots T_n T_1 \cdots T_{n-2} = T_2 \cdots T_n T_1 \cdots T_{n-1}
  \]
  holds.
  The first relation corresponds to the curves appearing
  in the cyclic order $\alpha_1, \dots, \alpha_n$
  when travelling in the counter-clockwise manner
  around $\Delta$, and the second corresponds
  to the other cyclic order.
\end{theorem}

The left-to-right direction is being taken care of by
Labruère's result, whereas Mortada's shows part
of the right-to-left direction.
In fact, the main topological insight
allowing us to prove Theorem~\ref{thm:cycle-relation}
is the following positive answer to Conjecture 5.5.5
in his thesis,
which asserts that the homomorphism
$\artin(\widetilde A_{n-1}) \to \mcg(S)$
mapping $s_i$ to $T_i$ is injective
whenever $S$ is a regular neighbourhood of
$\alpha_1 \cup \dots \cup \alpha_n$.
We prove the result by constructing suitable branched
coverings of annuli by regular neighbourhoods
of circuits in order to apply the Birman-Hilden theorem.

\begin{theorem}\label{thm:regular-nbhd}
  Let $S$ be a regular neighbourhood of a circuit
  $\alpha_1, \dots, \alpha_{n}$ of $n \geq 3$ curves.
  Then
  the subgroup of $\mcg(S)$ generated by the Dehn twists~$T_i$ about $\alpha_i$
  is geometrically isomorphic to $\artin(\widetilde A_{n-1})$.
\end{theorem}

A \emph{geometric embedding} of an Artin group into the
mapping class group of a surface is an injective homomorphism
mapping the standard generators to Dehn twists,
and a \emph{geometric isomorphism} is
a bijective geometric embedding
or its inverse.
It is an open question what Artin groups
geometrically embed into the mapping class group of a surface,
although partial answers are plentiful.
For instance,
the group generated by Dehn twists about
two curves intersecting
two or more times is isomorphic to a free group
on two generators.
Hence, the free group of two generators
(which is the Artin group $\artin(\widetilde A_1)$
by convention)
geometrically embeds~\cite[Theorem~3.14]{primer}.
More generally, free groups geometrically
embed~\cite[Theorem~1.1]{humphries-free}.
Perron-Vannier~\cite[Théorème~1]{perron-vannier}
showed that both $\artin(A_n)$ and $\artin(D_n)$
geometrically embed.

On the other hand,
because two Dehn twists can only generate
$\mathbb{Z}$, $\mathbb{Z}^2$,
a quotient of $\artin(A_2)$,
or a free group~\cite[Section~3.5.2]{primer},
the Artin group $\artin(\Gamma)$ does not
geometrically embed if $\Gamma$ contains
a weight different from $2, 3, \infty$.
Labruère showed that $\artin(\widetilde D_{n-1})$
does not geometrically embed~\cite[Theorem]{labruere}, where
$\widetilde D_{n-1}$ is the graph
\[
    \vcenter{\hbox{ \tikz[scale=0.6]{
      \draw (150:1) -- (0, 0) -- (210:1);
      \draw (0, 0) -- (1.5, 0);
      \draw (2.5, 0) -- (4, 0);
      \draw (4, 0) -- +(30:1) coordinate (A);
      \draw (4, 0) -- +(-30:1) coordinate (B);
      \draw[fill=white] (0, 0) circle (0.05);
      \draw[fill=white] (150:1) circle (0.05);
      \draw[fill=white] (210:1) circle (0.05);
      \draw[fill=white] (1, 0) circle (0.05);
      \node at (2, 0) {$\cdots$};
      \draw[fill=white] (3, 0) circle (0.05);
      \draw[fill=white] (4, 0) circle (0.05);
      \draw[fill=white] (A) circle (0.05);
      \draw[fill=white] (B) circle (0.05);
    }}}
\]
with $n$ vertices for $n \geq 5$,
and shortly after, Wajnryb showed that neither do
the exceptional groups
$\artin(E_6)$,
$\artin(E_7)$, $\artin(E_8)$~\cite[Theorem~3]{wajnryb}.

In the same article,
Wajnryb appears to claim that Labruère also showed that
$\artin(\widetilde A_{n-1})$
does not geometrically embed~\cite[Theorem~2]{wajnryb}.
But this is not true
(it contradicts Theorem~\ref{thm:regular-nbhd}).
Possibly, the confusion comes from Labruère considering
not a regular neighbourhood~$S$
of a family of curves intersecting
in a circuit pattern, but rather a surface
$S \cup \Delta$ containing an additional embedded closed disc
(in our notation of Section~\ref{ssec:building} below, Labruère
considers the surface $N^n \cup \Delta^1$).
By Theorem~\ref{thm:cycle-relation},
considering $S \cup \Delta$ instead of $S$
causes the introduction of a cycle relation.

Presumably, Theorem~\ref{thm:cycle-relation} could be proven
without any Artin group theory
by considering the action of Dehn twists
on curves in the surfaces, as the standard cycle relation
$T_n \cdots T_1 T_n \cdots T_3 = T_{n-1} \cdots T_1 T_n \cdots T_2$
is equivalent to the commutation relation
\[
  T_n \cdots T_2 T_3^{-1} \cdots T_n^{-1} \cdot T_1
  = T_1 \cdot T_n \cdots T_2 T_3^{-1} \cdots T_n^{-1}
\]
(see~\cite[Section~1]{baader-loenne} for a justification of this).
It thus suffices to prove that if
for the homeomorphism $h = T_n \cdots T_3$ the curves $h(\alpha_2)$
and $\alpha_1$ are disjoint (up to homotopy),
then the circuit $\alpha_1, \dots, \alpha_n$
bounds an embedded disc.
This idea is illustrated in Figure~\ref{fig:bigon},
where we consider a neighbourhood
(later called $N^4$) of the circuit in the case that
the cycle relation $T_4 T_3 T_2 T_1 T_4 T_3 = T_3 T_2 T_1 T_4 T_3 T_2$
holds.
The strategy of proving
the result pictorially has a few drawbacks, however.
Most notably, is not very convincing
to rely solely on such pictures, as the reasoning
for the presence of discs seems very prone to error,
in particular for some of the surfaces that are less easily
drawn in a flat way.
In the author's opinion, the approach taken in this text
is more insightful and reliable.

\begin{figure}[htb]
  \centering
  \begin{minipage}{0.33\textwidth}
    \centering
    \includegraphics{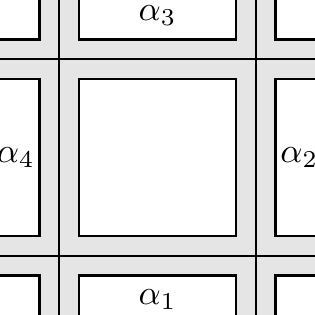}
  \end{minipage}%
  \begin{minipage}{0.33\textwidth}
    \centering
    \includegraphics{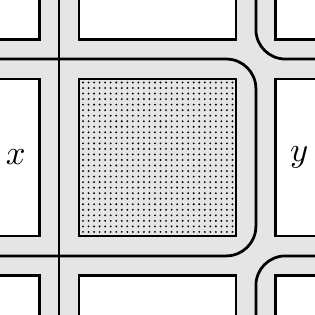}
  \end{minipage}%
  \begin{minipage}{0.33\textwidth}
    \centering
    \includegraphics{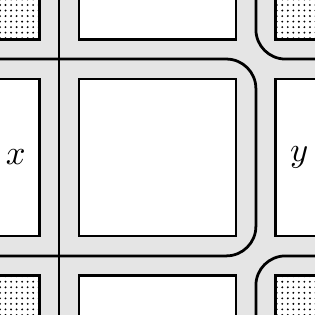}
  \end{minipage}%
  \caption{A neighbourhood of the curves $\alpha_1, \dots, \alpha_4$ and the resulting
    curves $x = \alpha$ and $y = T_3 T_2 (\alpha_1)$
    that must cobound one of the dotted bigons.
  }%
  \label{fig:bigon}
\end{figure}

A brief summary of sections is as follows.
Section~\ref{sec:artin}
is a collection of prerequisites from Artin group theory
and its relation to the theory of mapping class groups.
In Section~\ref{sec:nbhds}
we discuss possible regular neighbourhoods
of circuits and prove Theorem~\ref{thm:regular-nbhd}.
In our path toward a proof of Theorem~\ref{thm:cycle-relation}
we get sidetracked in Section~\ref{sec:discs}
and investigate the effect of
all (possibly non-embedded) discs
on the relational theory
of the group $\cgroup(S)$ generated by
the Dehn twists $T_1, \dots, T_n$ (note that $\cgroup(S)$ depends
not only on $S$ but also on the circuit $\alpha_1, \dots, \alpha_n$).
This allows us to
determine many isomorphism
types of $\cgroup(S)$ for the surfaces $S$ arising
from regular neighbourhoods
by gluing in discs,
see Table~\ref{tab:summary},
which is interesting in its own right.
The reader mainly interested in
Theorem~\ref{thm:cycle-relation}
may wish to focus on
Subsections~\ref{ssec:building},~\ref{ssec:extending},~\ref{ssec:cycles},
and skip the rest of Section~\ref{sec:discs}.
Finally, in Section~\ref{sec:punctured-discs-and-annuli},
we prove Theorem~\ref{thm:cycle-relation}.

\section{Artin groups and mapping class groups}\label{sec:artin}
\subsection{Artin groups}
To an undirected multigraph $\Gamma$
(i.e., a graph
with any number, infinity included, of edges between two vertices,
but no edge between a vertex and itself),
we associate the corresponding \emph{Artin group}
$\artin(\Gamma)$ by describing a presentation.
The generators of $\artin(\Gamma)$ are
the vertices of $\Gamma$,
and the relations are $sts \ldots = tst \ldots$,
where the words on both sides have length $n_{st} + 2$,
where $n_{st}$ is the number of edges between $s$ and~$t$.
The numbers $n_{st} + 2$ are sometimes
referred to as \emph{weights}.
Explicitly, if there is no edge between $s$ and~$t$
then they satisfy the
\emph{commutation relation} $st = ts$ (weight $2$),
if $s$ and~$t$ are joined by a single edge
then $s$ and~$t$ satisfy the \emph{braid relation}
$sts = tst$ (weight $3$),
and so on.
It is customary that $s$ and $t$ satisfying
no such relation is allowed.
This corresponds to infinitely many edges
joining $s$ and $t$ (weight $\infty$).
We call an Artin group $\artin(\Gamma)$ \emph{irreducible}
if $\Gamma$ is a connected graph.
A few example graphs $\Gamma$ are listed in Figure~\ref{fig:dynkin}.

\begin{figure}[htb]
  \centering
  \begin{minipage}{0.18\textwidth}
    \centering
    \includegraphics{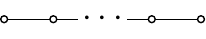}
  \end{minipage}%
  \begin{minipage}{0.21\textwidth}
    \centering
    \includegraphics{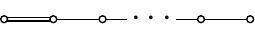}
  \end{minipage}%
  \begin{minipage}{0.20\textwidth}
    \centering
    \includegraphics{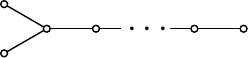}
  \end{minipage}%
  \begin{minipage}{0.21\textwidth}
    \centering
    \includegraphics{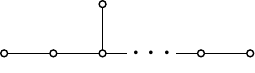}
  \end{minipage}%
  \begin{minipage}{0.18\textwidth}
    \centering
    \includegraphics{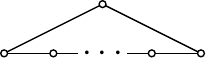}
  \end{minipage}%
  \caption{The graphs $A_{n}$, $B_{n}$, $D_{n}$, $E_{n}$, $\widetilde A_{n-1}$.
  All five graphs have $n$ vertices.}%
  \label{fig:dynkin}
\end{figure}

\subsection{The Birman exact sequence}
For a surface $S$,
we write $C(S, n)$ for the
configuration space of $n$ points in $S$~\cite[Section~9.1.2]{primer}.
In order to relate Artin groups to mapping class groups of
surfaces, a very useful tool is the Birman exact sequence.

\begin{lemma}[%
  {\cite[Theorem~9.1]{primer}}%
  ]%
  \label{lem:birman}
  Let $S$ be a surface without marked points
  such that the identity component
  of the group
  of orientation-preserving homeomorphisms $S \to S$
  keeping the boundary $\partial S$ fixed,
  is simply connected.
  Let $S_n$ be the surface obtained from $S$
  by marking $n$ points in the interior of $S$.
  Then there is an exact sequence
  \[
    1 \longrightarrow
    \pi_1(C(S, n)) \longrightarrow
    \mcg(S_n) \longrightarrow
    \mcg(S) \longrightarrow 1,
  \]
  where the homomorphism $\mcg(S_n) \to \mcg(S)$
  is obtained from forgetting that the marked points
  are marked.
\end{lemma}

The Birman exact sequence
is commonly used to prove that mapping class groups
of certain surfaces
are generated by finitely many Dehn twists~\cite[Theorem~4.1]{primer}.
Another application is the following classical Lemma.

\begin{lemma}[{\cite[Section~9.1]{primer}}]%
  \label{lem:a}
  Let $n \geq 2$, and
  let $\Delta$ be a closed disc and $\Delta_{n}$
  be $\Delta$ with $n$ marked points.
  Then the groups $\artin(A_{n-1})$, $\pi_1(C(\Delta, n))$, and
  $\mcg(\Delta_{n})$, are pairwise isomorphic.
\end{lemma}

Beware that the isomorphism between $\artin(A_{n-1})$ and
$\mcg(\Delta_{n})$ is not geometric,
because the generators are mapped to half-twists
(which are not Dehn twists).
It turns out that the Birman exact sequence
can also be applied to the annulus $Z$ rather than the
disc $\Delta$. This yields the following crucial result
for our work here.

\begin{lemma}
  \label{lem:annular}
  Let $n \geq 2$, and
  let $Z$ be an annulus and let $Z_n$ be $Z$ with
  $n$ marked points.
  Then the groups $\artin(B_n)$, $\pi_1(C(Z, n))$,
  and the
  kernel of the homomorphism
  $\mcg(Z_n) \to \mcg(Z)$
  forgetting that the marked points are marked,
  are pairwise isomorphic.
\end{lemma}

\begin{proof}
  In~\cite[Theorem One]{kent-peifer}, it is shown that
  $\artin(B_n)$ is isomorphic to the fundamental group
  $\pi_1 (C(Z, n))$
  of the configuration space of $n$ points in the annulus $Z$.
  Because the identity component of the
  space of
  orientation-preserving
  homeomorphisms of $Z$
  keeping the boundary fixed is contractible~\cite[Lemma~0.10]{scott},
  we can apply Lemma~\ref{lem:birman}
  to show that $\pi_1 (C (Z, n))$
  embeds into the mapping class group
  $\mcg(Z_n)$ of the annulus
  with $n$ marked points.
  More specifically, by exactness
  of the Birman exact sequence,
  the image of
  $\pi_1(C(Z, n))$ is the kernel of the
  homomorphism
  $\mcg(Z_n) \to \mcg(Z)$
  forgetting that the marked points are marked.
\end{proof}

\begin{remark}
The group $\pi_1(C(Z, n))$ can be thought of as
$n$-stranded braids in $Z \times [0, 1]$.
Projecting to the central curve of $Z$ at each height
in $[0, 1]$, but remembering which strand goes over
and which goes under, yields diagrams of elements
of $\pi_1(C(Z, n))$, similarly
as for the ordinary braid group $\pi_1(C(\Delta, n))$,
where one usually projects to a diameter of $\Delta$ at each height.
Let us refer to the generators of
$\artin(B_n)$ by the symbols $t, s_1, \dots, s_{n-1}$.
We also write $s_0 = \delta s_{n-1} \delta^{-1}$,
where $\delta = t s_1 \dots s_{n-1}$.
The isomorphism $\artin(B_n) \to \pi_1(C(Z, n))$
maps the elements $t, s_0, s_1, \dots, s_{n-1}$
to the annular braids depicted in Figure~\ref{fig:generators-b3}
for the case $n = 3$,
and stacks the braids from bottom to top when
reading the word in $\artin(B_n)$ from left to right.
\end{remark}

\begin{figure}[htb]
  \centering
  \begin{minipage}{0.25\textwidth}
    \centering
    \includegraphics{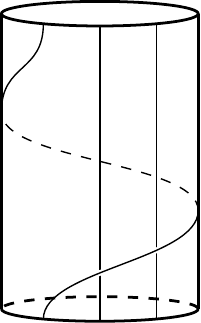}
  \end{minipage}%
  \begin{minipage}{0.25\textwidth}
    \centering
    \includegraphics{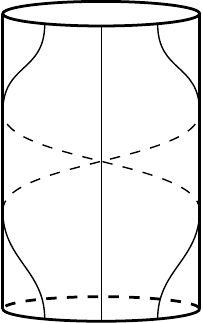}
  \end{minipage}%
  \begin{minipage}{0.25\textwidth}
    \centering
    \includegraphics{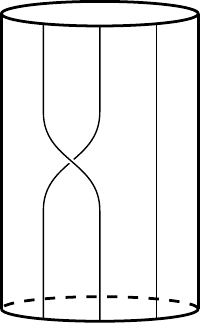}
  \end{minipage}%
  \begin{minipage}{0.25\textwidth}
    \centering
    \includegraphics{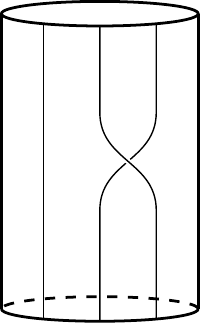}
  \end{minipage}%
  \caption{The images of the elements $t$, $s_0$, $s_1$, $s_2$ of
  $\artin(B_3)$ in $\pi_1(C(Z, 3))$}%
  \label{fig:generators-b3}
\end{figure}

\begin{remark}
  We can also explicitly describe the images of
  the elements $t, s_0, s_1, \dots, s_{n-1}$
  in $\mcg(\Delta_n)$.
  The element $t$ maps to a product
  $T_\alpha T_\beta^{-1}$
  of two Dehn twists about two curves
  $\alpha$ (inner) and  $\beta$ (outer)
  depicted in Figure~\ref{fig:generators-cylinder}.
  More importantly for us,
  the elements $s_i$ map to
  so-called
  \emph{half-twists}~\cite[Section~9.1.3]{primer}
  about the arcs depicted in Figure~\ref{fig:generators-cylinder}.
\end{remark}

\begin{figure}[htb]
  \centering
  \begin{minipage}{0.25\textwidth}
    \centering
    \includegraphics{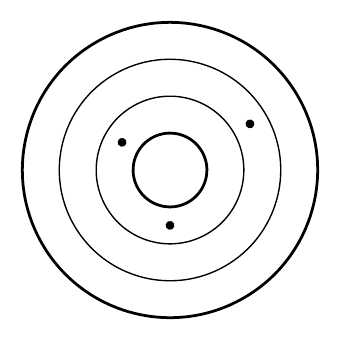}
  \end{minipage}%
  \begin{minipage}{0.25\textwidth}
    \centering
    \includegraphics{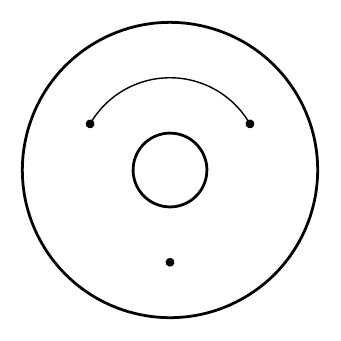}
  \end{minipage}%
  \begin{minipage}{0.25\textwidth}
    \centering
    \includegraphics{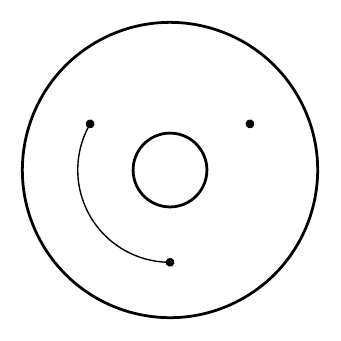}
  \end{minipage}%
  \begin{minipage}{0.25\textwidth}
    \centering
    \includegraphics{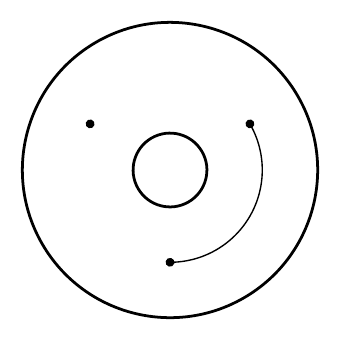}
  \end{minipage}%
  \caption{The images of the elements $t, s_0, s_1, s_2$ of $\artin(B_3)$
  in $\mcg(Z_n)$}%
  \label{fig:generators-cylinder}
\end{figure}

We are now ready to describe the group $\artin(\widetilde A_{n-1})$
appearing in Theorem~\ref{thm:regular-nbhd}
quite explicitly.

\begin{lemma}[%
  {\cite[Section~1]{charney-peifer}}]\label{lem:atilde}
  Let $t, s_1, \dots, s_{n-1}$ be the standard generators of
  $\artin(B_{n})$
  corresponding to the vertices
  in Figure~\ref{fig:dynkin} from left to right,
  and let
  $s_n = \delta s_{n-1} \delta^{-1}$,
  where $\delta = t s_1 \cdots s_{n-1}$.
  Then the subgroup $G$ of $\artin (B_{n})$
  generated by the $s_i$ is
  isomorphic to $\artin(\widetilde A_{n-1})$
  under an isomorphism
  mapping the $s_i$
  to the standard generators.
  Moreover, $G$ is the kernel of the homomorphism
  $\artin(B_{n}) \to \mathbb{Z}$ mapping~$t$ to one and
  all other generators to zero.
\end{lemma}

\subsection{Birman-Hilden theory}\label{ssec:birman-hilden}
In the 1970s, Birman and Hilden proved
highly influential results about fiber-preserving
isotopies.
Their work spawned an entire research area
referred to as Birman-Hilden theory,
recently surveyed by Margalit
and Winarski~\cite{margalit-winarski}.
Only a very small part of this theory
will find its way into this text.

For an orientation-preserving
involution $\iota$ on a surface $S$,
let us write $S/\iota$ for the quotient
of $S$ by $\iota$ with the images of the
fixed points of $\iota$ marked.
With this notation,
we have the following reformulation of
the classical Birman-Hilden theorem.

\begin{lemma}\label{lem:birman-hilden}
  Let the surface $S$
  have at least one boundary component,
  and let $\iota$ be a continuous involution on $S$
  with finitely many fixed points.
  Suppose that $\iota$ leaves the
  curves
  $\alpha_i$ invariant as sets
  and restricts to a reflection of each $\alpha_i$.
  Then there is a well-defined homomorphism
  $\cgroup(S) \to \mcg(S/\iota)$
  mapping the Dehn twists~$T_{i}$ about the $\alpha_i$
  to half-twists.
\end{lemma}

\begin{proof}
  Let $f$ be a symmetric homeomorphism of
  $S$, i.e., a homeomorphism
  that commutes with $\iota$.
  Then $f$ induces a homeomorphism $\overline f$ on
  the quotient $S/\iota$.
  The Birman-Hilden theorem~\cite[Theorem~1]{birman-hilden}
  asserts that
  the mapping class of $\overline f$
  does not depend on the symmetric representative of
  the mapping class of $f$.
  The Dehn twist $T_i$ about a simple closed curve $\alpha_i$
  is symmetric up to isotopy,
  so we obtain a homomorphism $\cgroup(S) \to \mcg(S/\iota)$.
  Moreover, $\iota$ restricts to an involution of an annular neighbourhood
  of $\alpha_i$ exchanging the two boundary components.
  Hence, $\overline T_i$ is a half-twist.
\end{proof}

It is worth pointing out the
nontrivial part of the proof of the result
referred to as the Birman-Hilden theorem
in Farb-Margalit's book~\cite[Section~9.4.1]{primer}
is Lemma~\ref{lem:birman-hilden},
formulated in a slightly different way~\cite[Proposition~9.4]{primer}.
The involutions $\iota$ they use yield well-defined maps
from the braid group $\artin(A_{n})$
on $n+1$ strands
to the group generated by Dehn twists
about a chain of $n$ curves (each
curve intersecting the previous and the next).

\section{Neighbourhoods of circuits}\label{sec:nbhds}
This section is concerned with the proof of
Theorem~\ref{thm:regular-nbhd}.
Let $\alpha_1, \dots \alpha_n$ be a circuit.
Up to orientation-preserving homeomorphism,
there are two possible regular neighbourhoods
of the union $\alpha_1 \cup \cdots \cup \alpha_n$,
see Figure~\ref{fig:nbhds}.
One way to see this is as follows.
There is only one possible
regular neighbourhood of the smaller set
$\alpha_1 \cup \cdots \cup \alpha_{n-1}$.
Now the curve $\alpha_n$ might sit in the regular neighbourhood
in two different ways.

\begin{figure}[htb]
  \centering
  \begin{minipage}{0.31\textwidth}
    \centering
    \includegraphics{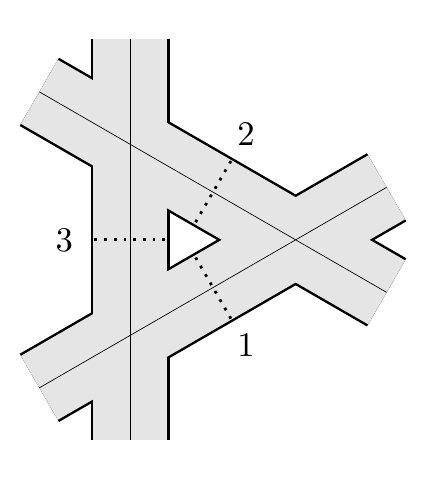}
  \end{minipage}%
  \begin{minipage}{0.34\textwidth}
    \centering
    \includegraphics{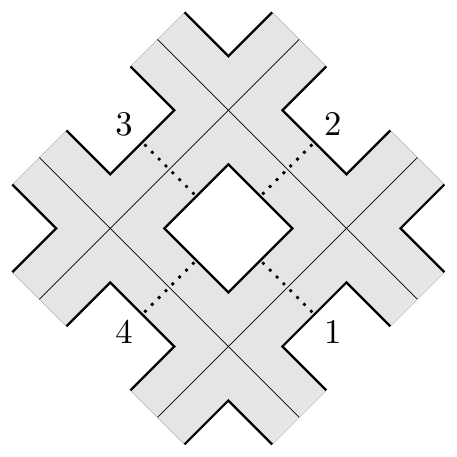}
  \end{minipage}%
  \begin{minipage}{0.34\textwidth}
    \centering
    \includegraphics{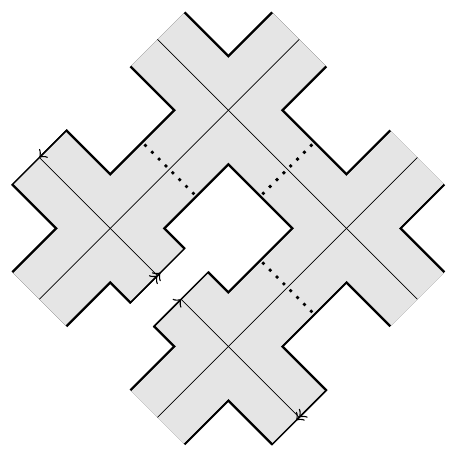}
  \end{minipage}%
  \caption{The neighbourhoods $N^3$, $N^4$, $M^4$.
  Opposite ends of the strips are identified,
unless indicated otherwise.}%
  \label{fig:nbhds}
\end{figure}

If $n$ is odd, those two ways lead to regular neighbourhoods
$N^{n}_{\circlearrowleft}$
and
$N^{n}_{\circlearrowright}$
that are related by an orientation-reversing homeomorphism.
For brevity, we will abbreviate
$N_{\circlearrowleft}^{n}$
by the symbol $N^{n}$
and usually not talk about $N^{n}_{\circlearrowright}$
explicitly,
as all the results about $N^{n}$ carry
over to $N^{n}_{\circlearrowright}$
by enumerating the $\alpha_i$ in the opposite order.
The left-hand side of Figure~\ref{fig:plastic-odd}
below
displays a drawing of
the surface $N^n$ embedded into $\mathbb{R}^3$.

If $n$ is even, the two possible
neighbourhoods $N^{n}$ and $M^{n}$ are
related to themselves via an orientation-reversing
homeomorphism, so orientation is less of a concern
in this case.
The neighbourhoods $N^{n}$ and $M^{n}$ differ,
for example, in their number of boundary components:
$N^{n}$ has four and $M^{n}$ just two.
A drawing of the surfaces $N^n$ and $M^n$
embedded into $\mathbb{R}^3$
can be found in Figure~\ref{fig:plastic-even}.

\begin{figure}[htb]
  \centering
  \begin{minipage}{0.52\textwidth}
    \centering
    \includegraphics{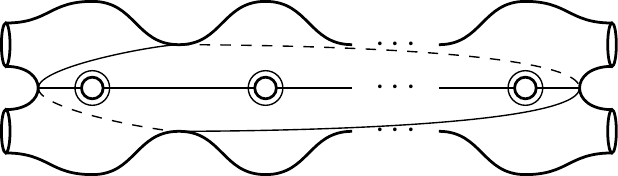}
  \end{minipage}%
  \begin{minipage}{0.48\textwidth}
    \centering
    \includegraphics{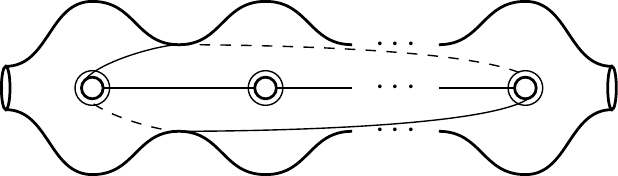}
  \end{minipage}%
  \caption{Another view of the surfaces $N^n$ and $M^n$ for even $n$}%
  \label{fig:plastic-even}
\end{figure}

Now let $S$ be any of the above regular neighbourhoods
of $\alpha_1 \cup \cdots \cup \alpha_n$.
Notice that in each case, $S$
can be thought of as a union of $n$ cross-shaped
pieces,
see Figure~\ref{fig:nbhds}.
Turning all those pieces by an angle of
$\pi$ yields a well-defined involution $\iota$ of $S$.
We will call~$\iota$ the \emph{cross-involution}.
Note that in the drawings from Figures~\ref{fig:plastic-even}
and~\ref{fig:plastic-odd}, the cross-involution
is a rotation about the $x$-axis by an angle of $\pi$.

\begin{proof}[Proof of Theorem~\ref{thm:regular-nbhd}]
  With the notation from Section~\ref{ssec:birman-hilden},
  the surface $S/\iota$ is an annulus
  with $n$ marked points
  in each case.
  This may be verified by counting the number
  of boundary components of $S/\iota$ and computing
  its Euler characteristic.
  By Lemma~\ref{lem:birman-hilden},
  there exists a homomorphism
  $\varphi \colon \cgroup(S) \to \mcg(Z_{n})$
  mapping the Dehn twists $T_i$ to half-twists.
  By Lemma~\ref{lem:atilde}, the image of $\varphi$
  is isomorphic to $\artin(\widetilde A_{n-1})$.
  But the inverse homomorphism $\artin(\widetilde A_{n-1})
  \to \cgroup(S)$
  mapping the generator $s_i$ to $T_i$ is well-defined.
  Hence, $\varphi$ is an isomorphism.
\end{proof}

\section{Gluing in discs}\label{sec:discs}
In this section, we glue in discs to regular
neighbourhoods of circuits in order to obtain
more geometric embeddings of Artin groups,
which we will need in order to prove
the right-to-left direction of
Theorem~\ref{thm:cycle-relation}.
In fact, the results in this section
are much stronger than what is needed.
Each possible combination of discs
that can be glued in to neighbourhoods
is investigated separately,
and a few interesting relations between Dehn twists
about circuits are discussed.

\subsection{Building circuit surfaces}\label{ssec:building}
We now adopt the perspective that the circuit
$\alpha_1, \dots, \alpha_n$ stays fixed
while the surface $S$ containing it varies.
We will write $\cgroup(S)$ for the subgroup
of $\mcg(S)$
generated by the Dehn twists $T_i$
about $\alpha_i$.
The inclusion homomorphism theorem asserts the following.

\begin{lemma}[{\cite[Theorem~3.18]{primer}}]\label{lem:inclusion-homo}
  Suppose $S$ and~$S'$ are closed
  and connected
  subsurfaces
  of a surface $S \cup S'$
  with disjoint interiors.
  Let $K$ be the kernel of the inclusion-induced
  homomorphism $\mcg(S) \to \mcg(S \cup S')$.
  Then:
  \begin{enumerate}[\normalfont(i)]
    \itemsep0em
    \item If $S' = \Delta_1$ is a once-marked disc
      with $\partial \Delta_1 \subset \partial S$,
      then $K$ is cyclically generated by
      the Dehn twist $T_\alpha$ about the
      boundary curve of $\Delta_1$.
    \item If $S' = Z$ is an annulus with $\partial Z \subset \partial S$,
      then $K$
      is cyclically generated by $T_\alpha T_\beta^{-1}$,
      where $T_\alpha$ and $T_\beta$ are Dehn twists
      about the boundary curves $\alpha$ and $\beta$
      of $Z$,
      respectively.
    \item If $S'$ is neither a disc,
      a once-marked disc,
      nor an annulus,
      then $K$ is trivial.
  \end{enumerate}
  In particular, $K$ is a subgroup
  of the center of $\mcg(S)$.
\end{lemma}

Note that the inclusion homomorphism theorem
makes no assertion in the case that $S'$
is a disc $\Delta$.
This means that to compare
$\cgroup(S \cup \Delta)$ to
$\cgroup(S)$ is expected to require
more creativity than to compare
$\cgroup(S \cup S')$
to $\cgroup(S)$
for other surfaces $S'$.

Not all boundary components
of neighbourhoods of $\alpha_1 \cup \cdots \cup \alpha_n$
are equivalent.
Considering intersection points between the $\alpha_i$
as vertices of polygons allows us to make the following
observation.
If $n$ is odd, then two boundary components
of $N^{n}$ are exchanged by the cross-involution.
They both have the property that if they are capped by a disc,
then $\alpha_1, \dots, \alpha_n$ bound an $n$-gon.
We will denote such a disc by $\Delta^1$,
and the union of two such discs by $2\Delta^1$.
Both of these will make the circuit
bound embedded closed discs.
The third boundary component can be capped off
by a disc~$\Delta^2$. The curves $\alpha_1, \dots, \alpha_n$
bound a $2n$-gon in $S \cup \Delta^2$.
Note that the circuit does not bound a closed
embedded disc in $N^n \cup \Delta^2$.
For brevity, we will sometimes say that
the boundary of $\Delta^1$ \emph{is} an $n$-gon
and that the boundary of $\Delta^2$
\emph{is}
a $2n$-gon.
See Figure~\ref{fig:surface-variants-odd}
for a
visual description
of the various discs.

\begin{figure}[htb]
  \centering
  \begin{minipage}{0.25\textwidth}
    \centering
    \includegraphics{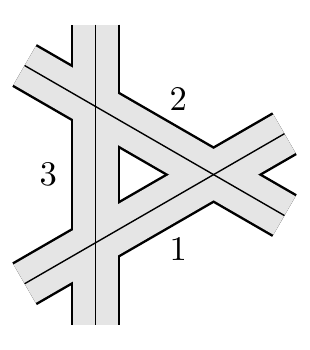}
  \end{minipage}%
  \begin{minipage}{0.25\textwidth}
    \centering
    \includegraphics{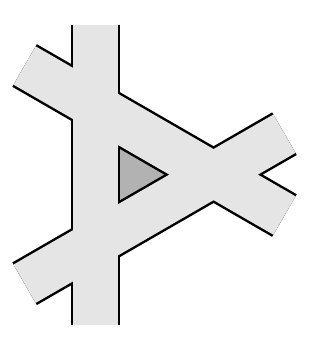}
  \end{minipage}%
  \begin{minipage}{0.25\textwidth}
    \centering
    \includegraphics{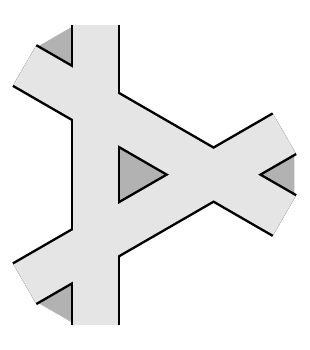}
  \end{minipage}%
  \begin{minipage}{0.25\textwidth}
    \centering
    \includegraphics{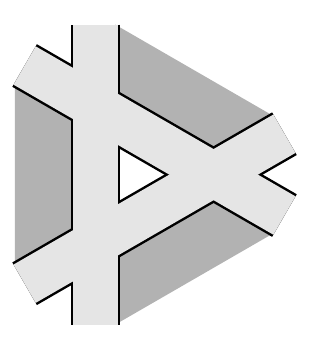}
  \end{minipage}%
  \caption{The surfaces $N^{3}$,
$N^{3} \cup \Delta^1$, $N^{3} \cup 2\Delta^1$, $N^{3} \cup \Delta^2$.
}%
  \label{fig:surface-variants-odd}
\end{figure}

If $n$ is even, the situation is qualitatively different.
All four boundary components of $N^{n}$ are
$n$-gons, and both boundary components
of $M^{n}$ are $2n$-gons.
But they differ in the following way.
Travelling around the boundary component in the counter-clockwise
direction, as seen from the center of the disc,
the curves might appear in the order $\alpha_1, \dots, \alpha_n$
or the other way around.
In the first case, we will write $\Delta_\circlearrowleft$,
and $\Delta_\circlearrowright$ otherwise,
with the appropriate superscript numbers.
See Figure~\ref{fig:surface-variants-even} for
the different discs in $N^{n}$.
We will usually abbreviate $\Delta_{\circlearrowleft}$
by $\Delta$ with the appropriate superscript number,
and we will abbreviate
$\Delta^2_\circlearrowleft \cup \Delta^2_\circlearrowright$
by $2\Delta^2$.
A list of all surfaces that are obtained by gluing in discs
to a regular neighbourhood of $\alpha_1 \cup \cdots \cup \alpha_n$
can be found in Table~\ref{tab:summary}.
Note that gluing in any disc into $N^n$ makes
the circuit bound an embedded closed disc,
and no disc in $M^n$ has this effect.

\begin{figure}[htb]
  \centering
  \begin{minipage}{0.25\textwidth}
    \centering
    \includegraphics{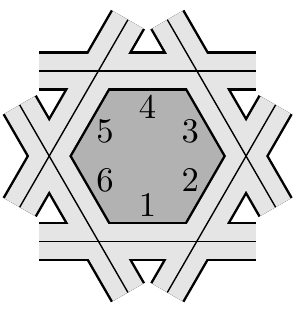}
  \end{minipage}%
  \begin{minipage}{0.25\textwidth}
    \centering
    \includegraphics{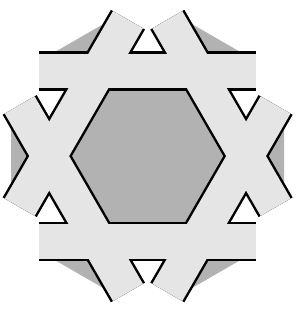}
  \end{minipage}%
  \begin{minipage}{0.25\textwidth}
    \centering
    \includegraphics{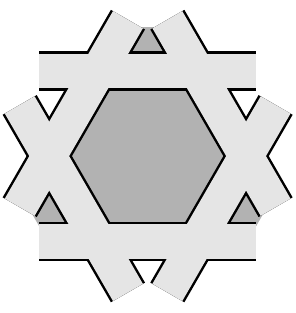}
  \end{minipage}%
  \begin{minipage}{0.25\textwidth}
    \centering
    \includegraphics{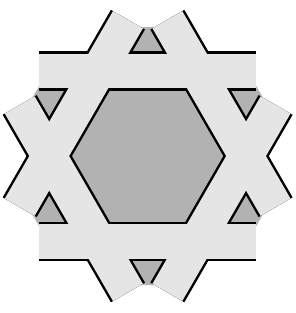}
  \end{minipage}%
  \caption{The surfaces $N^6 \cup \Delta^1_\circlearrowleft$,
  $N^6 \cup 2\Delta^1_\circlearrowleft$,
  $N^6 \cup \Delta^1_\circlearrowleft \cup \Delta^1_\circlearrowright$,
$N^6 \cup \Delta^1_\circlearrowleft \cup 2\Delta^1_\circlearrowright$}%
  \label{fig:surface-variants-even}
\end{figure}

\begin{table}[htb]
  \centering
  \begin{tabular}{lll}
    \toprule
    $S$ & $\cgroup(S)$ & Reference \\
    \midrule
    $N^n$ & $\artin(\widetilde A_{n-1})$ & Theorem~\ref{thm:regular-nbhd} \\
    $N^{n} \cup \Delta^1$ & $\artin(D_{n})$ &
    Proposition~\ref{prop:dn} \\
    $N^n \cup 2 \Delta^1$ & $\artin(A_{n-1})$ &
    Proposition~\ref{prop:an} \\
    \midrule
    $N^{2k + 1} \cup \Delta^2$ & $\artin(\widetilde A_{2k})$ &
    Proposition~\ref{prop:antilde} \\
    $N^{2k + 1} \cup \Delta^1 \cup \Delta^2$ & not Artin &
    Proposition~\ref{prop:special-disc} \\
    $N^{2k + 1} \cup 2\Delta^1 \cup \Delta^2$ & not Artin &
    Proposition~\ref{prop:odd-closed} \\
    \midrule
    $N^{2k+4} \cup \Delta^1_\circlearrowleft \cup \Delta^1_\circlearrowright$
      & not Artin & Proposition~\ref{prop:not-artin6} \\
      $N^{2k+4} \cup 2\Delta^1_\circlearrowleft \cup \Delta^1_\circlearrowright$
      & not Artin & Proposition~\ref{prop:not-artin6} \\
      $N^{2k+4} \cup 2\Delta^1_\circlearrowleft \cup 2\Delta^1_\circlearrowright$
      & not Artin & Proposition~\ref{prop:not-artin6} \\
    \midrule
    $N^4 \cup \Delta^1_\circlearrowleft \cup \Delta^1_\circlearrowright$ & $\artin (A_3)$ &
    Proposition~\ref{prop:4} \\
    $N^4 \cup 2\Delta^1_\circlearrowleft \cup \Delta^1_\circlearrowright$ & $\artin(A_2)$ &
    Proposition~\ref{prop:4} \\
    $N^4 \cup 2\Delta^1_\circlearrowleft \cup 2\Delta^1_\circlearrowright$ & $\specli(2, \mathbb{Z})$ &
    Proposition~\ref{prop:4} \\
    \midrule
    $M^{2k+2}$ & $\artin(\widetilde A_{2k+1})$ & Theorem~\ref{thm:regular-nbhd} \\
    $M^{2k+2} \cup \Delta^2$ & $\artin(\widetilde A_{2k+1})$ & Proposition~\ref{prop:antilde} \\
    $M^{2k+2} \cup 2\Delta^2$ & $\artin(\widetilde A_{2k+1})$ & Proposition~\ref{prop:antilde} \\
    \bottomrule
  \end{tabular}
  \bigskip
  \caption{%
    Various surfaces $S$ obtained from regular neighbourhoods
    of a circuit $\alpha_1, \dots, \alpha_n$ by gluing in discs.
    In this table, $n \geq 3$ and $k \geq 1$.
    Moreover, ``not Artin'' is short for
    ``not geometrically isomorphic to an Artin group''.}%
  \label{tab:summary}
\end{table}

\subsection{Extending the cross-involution}\label{ssec:extending}
It turns out that for certain surfaces $S$ containing
a neighbourhood of $\alpha_1, \dots, \alpha_{n}$,
the cross-involution
of the neighbourhood
extends to $S$.
In these cases, by very similar reasoning as in the proof
of Theorem~\ref{thm:regular-nbhd}, we are able to determine
the group $\cgroup(S)$.

\begin{proposition}\label{prop:antilde}
  Let $S$ be a regular neighbourhood
  of a circuit of $n \geq 3$ curves $\alpha_1, \dots \alpha_n$,
  and let $S'$ be
  a $2n$-gon $\Delta^2$ (such a disc does not exist if $n$ is even
  and $S = N^n$),
  or the union $2\Delta^2$ of two $2n$-gons
  (such a union only exists for even~$n$ and $S = M^{n}$).
  Then $\cgroup(S \cup S')$
  is geometrically isomorphic to $\artin(\widetilde A_{n-1})$.
\end{proposition}

\begin{proof}
  In each case, the cross-involution $\iota$ extends
  to $S \cup S'$
  in a straightforward fashion.
  In the case that $S' = \Delta^2$ is a $2n$-gon,
  $\iota$ gets one additional fixed point, so
  $(S \cup S') / \iota$ is a disc $\Delta_{n+1}$
  with $n+1$ marked points.
  Again, the homomorphism $\varphi \colon \cgroup(S) \to \mcg(\Delta_{n+1})$
  is well-defined by Lemma~\ref{lem:birman-hilden}.
  The images of the $T_i$ under $\varphi$ fix the last puncture.
  Because $\Delta$ minus one point
  is homotopy equivalent to the annulus $Z$,
  we have that the subgroup of $\pi_1(C(\Delta, n+1))$
  fixing one strand is isomorphic to
  $\pi_1(C(Z, n))$.
  By Lemmas~\ref{lem:annular} and~\ref{lem:atilde},
  the image of $\varphi$ is isomorphic to $\artin(\widetilde A_{n-1})$.
  Since the inverse homomorphism is well-defined, the result follows
  for this case.

  Similarly, if $S' = 2\Delta^2$,
  we get a well-defined homomorphism
  $\varphi \colon \cgroup(S) \to \mcg(\Sigma_{n + 2})$,
  where $\Sigma_{n+2}$ is a sphere with
  $n + 2$ marked points.
  The images of the $T_i$ fix two of the punctures.
  Because the group of orientation-preserving
  homeomorphisms of $\Sigma$ is not simply connected,
  there is no Birman exact sequence.
  That is,
  we cannot apply Lemma~\ref{lem:birman} directly
  to get a description of $\mcg(\Sigma_{n+2})$.
  However, it is straightforward to
  show that the kernel of the map
  $\pi_1(C(\Sigma, n+ 2)) \to \mcg(\Sigma_{n+2})$
  is generated by the map rotating the $n + 2$
  marked points by a full twist~\cite[Section~9.1]{primer}.
  Because such a full twist cannot be expressed by just $n$
  generators, we have that the image of $\varphi$
  is isomorphic to the subgroup of
  $\pi_1(C(\Sigma, n+2))$ fixing two strands.
  But because $\Sigma$
  minus two points is homotopy equivalent
  to $Z$, we get that the image of $\varphi$ is isomorphic
  to $\artin(\widetilde A_{n-1})$, as desired.
\end{proof}

Recall that for odd $n$ we abbreviate $N_\circlearrowleft^n$
by $N^n$, and for all $n$,
we abbreviate $\Delta_\circlearrowleft^1$ by $\Delta^1$.
Using these conventions allows us to concisely state the following
result which
is essentially equivalent
to a version of
the Birman-Hilden theorem stated
in the book by Farb and Margalit~\cite[Theorem~9.2]{primer}.

\begin{proposition}\label{prop:an}
  For $n \geq 3$, the group $\cgroup(N^n \cup 2\Delta^1)$ is geometrically isomorphic
  to the braid group $\artin(A_{n-1})$ on $n$ strands.
\end{proposition}

\begin{proof}
  Let us write $S = N^n \cup 2\Delta^1$
  The cross-involution $\iota$ extends
  to $N^n \cup 2\Delta^1$ with no additional fixed points.
  This yields a well-defined homomorphism
  $\varphi \colon \cgroup(S) \to \mcg(\Delta_n)$
  mapping the $T_i$ to half-twists
  by Lemma~\ref{lem:birman-hilden}.
  Since these half-twists generate $\mcg(\Delta_n)$,
  Lemma~\ref{lem:a} yields that
  the image of $\varphi$ is isomorphic
  to $\artin(A_{n-1})$.
\end{proof}

\begin{remark}
  It is not difficult to show that the kernel
  of the inclusion-induced homomorphism
  $\cgroup(N^n) \to \cgroup(N^n \cup 2\Delta^1)$
  is normally
  generated by the relation $T_n \cdots T_2 = T_{n-1} \cdots T_1$.
  One possible strategy is to explicitly compute the kernel of
  the inclusion-induced
  homomorphism $\pi_1(C(Z, n)) \to \pi_1(C(\Delta, n))$.
\end{remark}

\subsection{The cycle relation}\label{ssec:cycles}
The surface the current subsection is about
is $N^n \cup \Delta^1$.
Luckily for us, this case has almost entirely been
solved by Labruère, and the rest can be extracted
from work by Baader and Lönne.

If the circuit $\alpha_1, \dots, \alpha_n$ bounds an $n$-gon
$\Delta^1_{\circlearrowleft}$
or $\Delta^1_{\circlearrowright}$,
we will say that $\alpha_1, \dots, \alpha_n$
form a \emph{cycle}.
The \emph{standard cycle relation} between the Dehn twists
$T_1, \dots, T_n$ is
\(
  T_n \cdots T_1 T_n \cdots T_3 = T_{n-1} \cdots T_1 T_n \cdots T_2.
\)
One can show that this relation is equivalent to the commutation relation
$T_1 f = f T_1$ where
$f = (T_n \cdots T_3) T_2 {(T_n \cdots T_3)}^{-1}$~\cite[Section~1]{baader-loenne}.
Using this representation of the standard cycle relation it becomes
a routine task to verify that the standard cycle relation holds
in the surface $N^n \cup \Delta^1$.
Similarly, the \emph{reverse cycle relation}
\(
  T_1 \cdots T_n T_1 \cdots T_{n-2} = T_{2} \cdots T_n T_1 \cdots T_{n-1}.
\)
holds in the surface $N^n_{\circlearrowright} \cup \Delta^1_{\circlearrowright}$.
But Labruère made an even stronger observation.

\begin{lemma}[{\cite[Proposition~2]{labruere}}]\label{lem:labruere}
  The kernel of the homomorphism $\artin(\widetilde A_{n-1})
  \to \cgroup(N^n \cup \Delta^1)$ mapping the standard generators $s_i$
  to $T_i$ is normally generated by the cycle relation.
\end{lemma}

\begin{proposition}\label{prop:dn}
  For $n \geq 3$,
  the group $\cgroup(N^n \cup \Delta^1)$
  is geometrically isomorphic to $\artin(D_n)$.
\end{proposition}

\begin{proof}
  Let $s_1, \dots, s_n$ be the standard generators of
  $\artin(D_n)$ read from left to right
  in Figure~\ref{fig:dynkin}.
  Using Lemma~\ref{lem:labruere}, one can verify that
  an explicit isomorphism $\artin(D_n) \to \cgroup(N^n \cup \Delta^1)$
  is given by $s_1 \mapsto {(T_n \cdots T_3)}^{-1} T_1 (T_n \cdots T_3)$
  and $s_i \mapsto T_i$ for $i \geq 2$.
\end{proof}

\begin{remark}
  Baader and Lönne prove the considerably more general
  but also less easily digested result that
  the secondary braid group is invariant via a
  geometric isomorphism under
  elementary conjugation~\cite[Section~4]{baader-loenne}.
  Indeed, by Lemma~\ref{lem:labruere},
  the group $\cgroup(N^n \cup \Delta^1)$ is geometrically
  isomorphic to
  the secondary braid group~\cite[Definition~1]{baader-loenne}
  associated to the positive braid word
  $\sigma_1 \sigma_2 \sigma_1^{n-2} \sigma_2 \sigma_1$,
  whereas $\artin(D_n)$ is
  geometrically isomorphic to the group associated to
  $\sigma_1^2 \sigma_2 \sigma_1^{n-2} \sigma_2$.
\end{remark}

\subsection{Inhomogeneous relations}\label{ssec:inhomogeneous}
Sadly, we do not manage to compute the remaining groups $\cgroup(S)$
up to isomorphism.
We will, however, get to know the groups well enough
to exclude the possibility of them being geometrically isomorphic
to an Artin group.

A relation $t = t'$ is called \emph{homogeneous}
if the exponent sums of $t$ and $t'$ agree.
Otherwise, the relation $t = t'$ is called \emph{inhomogeneous}.
The strategy in the current subsection will be to
find inhomogeneous relations in $\cgroup(S)$.
The following elementary result allows us to conclude
that $\cgroup(S)$ is not geometrically isomorphic to an Artin group.

\begin{lemma}\label{lem:inhomogeneous}
  If $\cgroup(S)$ has an inhomogeneous relation,
  then $\cgroup(S)$ is not
  geometrically isomorphic to
  an Artin group.
\end{lemma}

\begin{proof}
  We argue contrapositively: The map sending each generator
  of an Artin group $\artin(\Gamma)$ to one
  extends to a homomorphism $\artin(\Gamma) \to \mathbb{Z}$
  because all relations in $\artin(\Gamma)$ are homogeneous,
  and hence also hold in $\mathbb{Z}$.
\end{proof}

An effective way to produce inhomogeneous relations
is to apply the classical result called
the chain relation.
Recall that a \emph{chain} of curves
is a family $\alpha_1, \dots, \alpha_n$ of $n$ curves
such that $\alpha_i$ intersects $\alpha_{j}$
exactly once if $j = i \pm 1$ and zero times otherwise.

\begin{lemma}[{\cite[Proposition~4.12]{primer}}]\label{lem:chain-relation}
  Let $\alpha_1, \dots, \alpha_n$ be a chain of $n$ curves,
  and let $T_i$ be the Dehn twist about $\alpha_i$.
  If $n$ is even, let $\beta$ be the boundary curve of a regular
  neighbourhood of $\alpha_1 \cup \cdots \cup \alpha_n$.
  Similarly, if $n$ is odd, let $\beta_1, \beta_2$ be the two
  boundary curves.
  Then:
  \begin{enumerate}[\normalfont(i)]
    \itemsep0em
    \item If $n$ is even, then ${(T_n \cdots T_1)}^{2n+2} = T_\beta$.
    \item If $n$ is odd, then ${(T_n \cdots T_1)}^{n+1} = T_{\beta_1} T_{\beta_2}$.
  \end{enumerate}
\end{lemma}

\begin{proposition}\label{prop:odd-closed}
  For odd $n \geq 3$,
  the relation ${(T_{n-1} \cdots T_{1})}^{2n} = 1$
  holds in
  the group $\cgroup(N^n \cup 2\Delta^1 \cup \Delta^2)$.
  In particular, it
  is not geometrically isomorphic to an Artin group.
\end{proposition}

\begin{proof}
  Note that the boundary of
  a regular neighbourhood of the chain $\alpha_1, \dots, \alpha_{n-1}$
  is null-homotopic.
  The proposition now follows from Lemma~\ref{lem:chain-relation}
  and Lemma~\ref{lem:inhomogeneous}.
\end{proof}

\begin{proposition}\label{prop:not-artin6}
  The group $\cgroup(N^n \cup \Delta^1_\circlearrowleft \cup \Delta^1_\circlearrowright)$
  is not geometrically isomorphic to an Artin group
  if $n \geq 6$, and neither is $\cgroup(S)$ for any supersurface
  $S$ of $N^n \cup \Delta^1_\circlearrowleft \cup \Delta^1_\circlearrowright$.
\end{proposition}

\begin{proof}
  Suppose the curves are arranged as in Figure~\ref{fig:two-discs}.
  Let $\beta = T_{n} \cdots T_3 T_3 \cdots T_{n}(\alpha_1)$.
  Then $\beta \cup \alpha_1$ is the boundary
  of a regular neighbourhood of
  $\alpha_3 \cup \cdots \cup \alpha_{n-1}$,
  see Figure~\ref{fig:two-discs}.
  By the Lemma~\ref{lem:chain-relation},
  the relation $(T_{n-1} \cdots T_3)^{n-2} = T_1 T_\beta$
  follows.
  Because $T_\beta$ is conjugate to $T_1$
  by the formula $T_{f(\alpha_1)} = f T_1 f^{-1}$~\cite[Fact~3.7]{primer}
  for $f = T_n \cdots T_3 T_3 \cdots T_n$,
  the relation in question is inhomogeneous for $n \geq 6$.
  Lemma~\ref{lem:inhomogeneous} leads us to the desired conclusion.

  If the indices of the curves are shifted by one from
  the ones in Figure~\ref{fig:two-discs},
  we instead end up with the relation ${(T_{n-2} \cdots T_2)}^{n-2} = T_n T_{\beta'}$
  where $\beta' = T_{n-1} \cdots T_2 T_2 \cdots T_{n-1} (\alpha_n)$,
  which is also inhomogeneous for $n \geq 6$.

  The statement about supersurfaces follows from the fact
  that the inclusion-induced homomorphisms preserve
  inhomogeneous relations.
\end{proof}

\begin{figure}[htb]
  \centering
  \begin{minipage}{0.33\textwidth}
    \centering
    \includegraphics{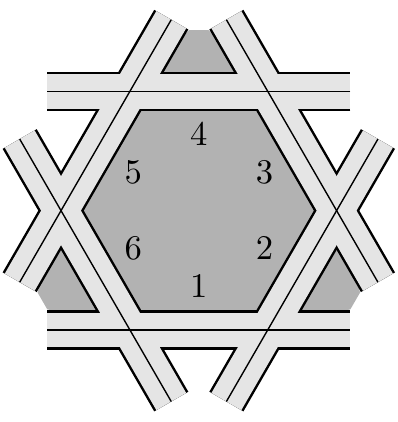}
  \end{minipage}%
  \begin{minipage}{0.33\textwidth}
    \centering
    \includegraphics{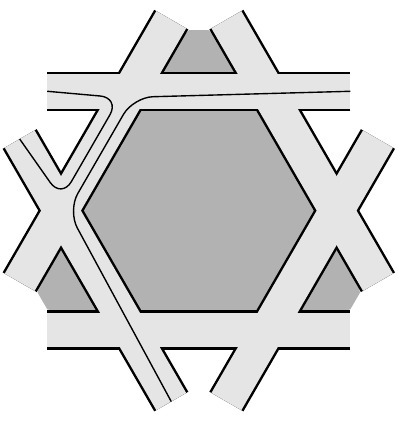}
  \end{minipage}%
  \begin{minipage}{0.33\textwidth}
    \centering
    \includegraphics{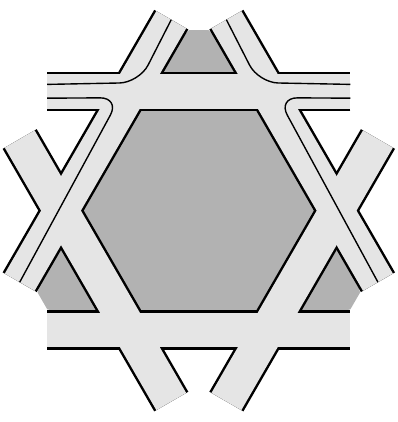}
  \end{minipage}%
  \caption{The curves $\alpha_1, \dots, \alpha_6$,
the curve $T_3 T_4 T_5 T_6 (\alpha_1)$,
and the curve $\beta = T_6 T_5 T_4 T_3^2 T_4 T_5 T_6 (\alpha_1)$,
all in
the surface $N^6 \cup \Delta^1_\circlearrowleft \cup \Delta^1_\circlearrowright$
  }%
  \label{fig:two-discs}
\end{figure}

\subsection{Pathological cases}
The case $n = 4$ becomes strange when too
many discs are glued in, because
some of the curves become isotopic.
The relations from the proof of
Proposition~\ref{prop:not-artin6}
do not reflect this, so we
cover this case separately.

\begin{proposition}\label{prop:4}
  The following statements hold.
  \leavevmode
  \begin{enumerate}[\normalfont(i)]
    \itemsep0em
    \item The group $\cgroup(N^4 \cup \Delta^1_\circlearrowleft \cup \Delta^1_\circlearrowright)$
      is geometrically isomorphic to $\artin(A_3)$.
      \item The group $\cgroup(N^4 \cup 2\Delta^1_\circlearrowleft \cup \Delta^1_\circlearrowright)$
        is geometrically isomorphic to $\artin(A_2)$,
      \item The group $\cgroup(N^4 \cup 2\Delta^1_\circlearrowleft \cup 2\Delta^1_\circlearrowright)$
        is isomorphic to $\specli(2, \mathbb{Z})$, and not geometrically isomorphic to an Artin group.
  \end{enumerate}
\end{proposition}

\begin{proof}
  Suppose the curves and discs are arranged as in Figure~\ref{fig:4}.
  We consider each surface~$S$ separately.

  \begin{enumerate}[(i)]
    \itemsep0em
    \item Let $S = N^4 \cup \Delta_\circlearrowleft^1 \cup \Delta_\circlearrowright^1$ Because $\alpha_2$ and $\alpha_4$ are isotopic
      in $S$,
      we have that $\cgroup(S)$ is
      generated by $T_1, T_2, T_3$.
      Moreover, $S$ is a regular neighbourhood
      of $\alpha_1, \alpha_2, \alpha_3$, so $\cgroup(S)$
      is isomorphic to $\artin(A_3)$~\cite[Section~9.4.1]{primer}.
    \item Let $S = N^4 \cup 2\Delta^1_\circlearrowleft \cup \Delta^1_\circlearrowright$. In addition to $\alpha_2$ being
      isotopic to $\alpha_4$ from the previous case,
      $\alpha_1$ is also isotopic to $\alpha_3$.
      So $\cgroup(S)$ is generated by $T_1, T_2$.
      Moreover, $S$ is a regular neighbourhood of $\alpha_1 \cup \alpha_2$.
      Hence, $\cgroup(S)$ is isomorphic to $\artin(A_2)$~\cite[Theorem~9.2]{primer}.
    \item Let $S = N^4 \cup 2\Delta^1_\circlearrowleft \cup 2\Delta^1_\circlearrowright$.
      Then $S$ is just a torus with meridian $\alpha_1$
      and $\alpha_2$. It is well-known that
      $\mcg(S)$ is generated by $T_1, T_2$,
      and that it is isomorphic to
      $\specli(2, \mathbb{Z})$~\cite[Theorem~2.5]{primer}.
      Moreover, the inhomogeneous relation
      ${(T_1 T_2)}^6 = 1$
      (see~\cite[Section~3.5]{primer}) shows that $\cgroup(S)$ is not geometrically
      isomorphic to an Artin group.
  \end{enumerate}
  If the indices of the curves are instead shifted by one,
  the same arguments hold.
\end{proof}

\begin{figure}[htb]
  \centering
  \begin{minipage}{0.25\textwidth}
    \centering
    \includegraphics{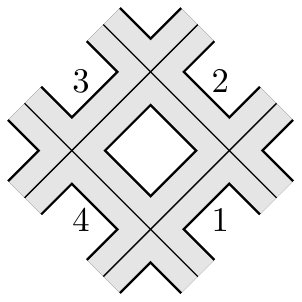}
  \end{minipage}%
  \begin{minipage}{0.25\textwidth}
    \centering
    \includegraphics{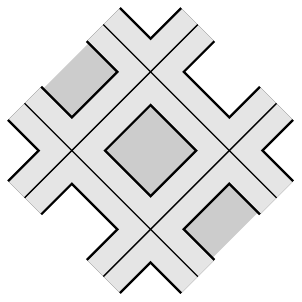}
  \end{minipage}%
  \begin{minipage}{0.25\textwidth}
    \centering
    \includegraphics{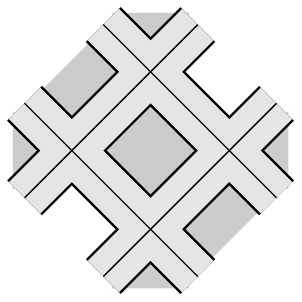}
  \end{minipage}%
  \begin{minipage}{0.25\textwidth}
    \centering
    \includegraphics{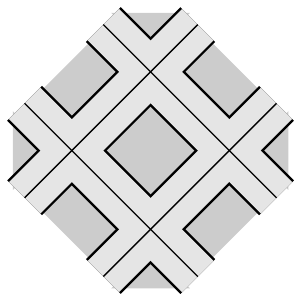}
  \end{minipage}%
  \caption{The surfaces $N^4$, $N^4 \cup \Delta_\circlearrowleft^1 \cup \Delta^1_\circlearrowright$,
  $N^4 \cup 2\Delta^1_\circlearrowleft \cup \Delta^1_\circlearrowright$,
$N^4 \cup 2\Delta^1_\circlearrowleft \cup 2\Delta^1_\circlearrowright$}%
  \label{fig:4}
\end{figure}

\subsection{One last surface}
Up to orientation-reversing homeomorphism,
we have now glued in every possible combination
of discs, except one.
For this final surface $S = N^{n} \cup \Delta^1 \cup \Delta^2$,
the strategy of finding inhomogeneous relations
failed, so the proof that $\cgroup(S)$
is not geometrically isomorphic to an Artin group
turns out to be the most involved argument in this text.
Toward a contradiction, we will assume
that $\cgroup(S)$
is geometrically isomorphic to
an Artin group $\artin(\Gamma)$.
We then exclude all possibilities for the graph $\Gamma$.
Lemma~\ref{lem:reflection-quotients-dn}
below is a statement about Coxeter groups
that helps achieve this
for most graphs.

The \emph{Coxeter group} $\cox(\Gamma)$ is obtained
from $\artin(\Gamma)$ by adding the relations
$s^2 = 1$ for all generators $s$.
If $\cox(\Gamma)$ is finite,
we will say that $\artin(\Gamma)$ is of
\emph{finite type}.
Otherwise, $\artin(\Gamma)$ is of \emph{infinite type}.
The finite Coxeter groups were classified
by Coxeter himself~\cite[Theorem\ddag]{coxeter}.
They are groups of the form $\cox(\Gamma)$,
where $\Gamma = A_n, B_n, D_n$ for arbitrary $n$,
$\Gamma = E_n$ for $n = 6, 7, 8$,
or a few more graphs that do not appear in this text.
See Figure~\ref{fig:dynkin} for a list of the
mentioned graphs.

\begin{lemma}[{\cite[Theorem~0.4 and Table~3]{maxwell}}]\label{lem:reflection-quotients-dn}
  Let $n \geq 3$ with $n \ne 4$.
  If there exists a surjective
  homomorphism
  $\cox(D_n) \to \mathcal{C}(\Gamma)$,
  then $\Gamma$ is either the one-vertex graph $A_1$,
  the graph $A_{n-1}$, or the graph $D_n$.
\end{lemma}

Next, we need two results
about the group-theoretic properties of the Artin groups
$\artin(A_{n-1})$ and $\artin(D_n)$.
The first result asserts that
$\artin(A_{n-1})$ is
``geometrically co-Hopfian''.

\begin{lemma}\label{lem:almost-co-hopfian}
  Let $n \geq 2$.
  Every injective homomorphism $\artin(A_{n-1}) \to \artin(A_{n-1})$
  such that the image of a standard generator is conjugate to a standard
  generator is an isomorphism.
\end{lemma}

\begin{proof}
  Think of $\artin(A_{n-1})$ as
  the group $\cgroup(S)$
  generated by $\alpha_1, \dots, \alpha_{n-1}$,
  where $S$ is the surface $N^n \cup 2\Delta^1$,
  see Proposition~\ref{prop:an}.
  Then a homomorphism as in the assumption corresponds to
  an injective homomorphism $\varphi \colon \cgroup(S) \to \cgroup(S)$
  mapping each $T_i$ to a Dehn twist $T_i'$
  about a curve $\alpha_i'$.
  Because $\varphi$ is injective,
  the $\alpha_i'$ are pairwise non-isotopic.
  Moreover, the curves $\alpha_1', \dots, \alpha_{n-1}'$
  form a chain because consecutive curves
  satisfy the braid relation~\cite[Section~3.5.2]{primer}.
  Hence, by the change of coordinates principle~\cite[Section~1.3.3]{primer},
  there exists a homeomorphism $f$ of $S$ such that
  $\alpha_i' = f(\alpha_i)$.
  Thus, $\varphi$ is given by conjugation by $f$,
  and hence is an isomorphism.
\end{proof}

\begin{remark}
  Bell and Margalit in fact describe all
  the injective homomorphisms
  from the $n$-strand braid group $\artin(A_{n-1})$ to itself,
  even the non-geometric ones,
  for $n \geq 4$~\cite[Main Theorem~1]{bell-margalit-co-hopfian}.
  Their uniform description of these homomorphisms
  does not hold for $n = 3$ because $\artin(A_2)$
  modulo its center
  is not co-Hopfian (it is isomorphic to the free product
  $\mathbb{Z}/2 * \mathbb{Z}/3$).
\end{remark}

Our final lemma in this Section asserts
that finite type Artin groups are ``Hopfian''.

\begin{lemma}\label{lem:hopfian}
  Every surjective homomorphism from a finite type Artin group
  onto itself is an isomorphism.
\end{lemma}

\begin{proof}
  Because finite type Artin groups
  are residually finite~\cite[Corollary~1.2]{blasco-garcia}
  they are also Hopfian~\cite[Theorem~IV.4.10]{lyndon-schupp}.
\end{proof}

\begin{proposition}\label{prop:special-disc}
  For odd $n \geq 3$,
  the group $\cgroup(N^{n} \cup \Delta^1 \cup \Delta^2)$
  is not geometrically isomorphic to an Artin group.
\end{proposition}

\begin{proof}
  Write $S = N^n \cup \Delta^1 \cup \Delta^2$.
  Suppose toward a contradiction
  that $\cgroup(S)$ is geometrically isomorphic to $\artin(\Gamma)$
  for a graph $\Gamma$.
  Recall that by Proposition~\ref{prop:dn}, the group
  $\cgroup(N^n \cup \Delta^1)$ is geometrically isomorphic
  to $\artin(D_n)$.
  Thus, the inclusion-induced homomorphism
  $\cgroup(N^n \cup \Delta^1) \to \cgroup(S)$
  gives rise to a surjective
  homomorphism
  $\cox(D_n) \to \mathcal{C}(\Gamma)$
  (note that we use here that the isomorphism $\artin(D_n)
  \to \cgroup(N^n \cup \Delta^1)$ is geometric).
  From Lemma~\ref{lem:reflection-quotients-dn}
  it follows that $\Gamma$ is either $A_1$, $A_{n-1}$, or $D_n$.
  We will now rule out each of those graphs.

  We first argue that $\cgroup(S)$ contains a strict
  subgroup isomorphic to $\artin(A_{n-1})$.
  Consider the
  plastic view of $N^n$
  as on the left of Figure~\ref{fig:plastic-odd}.
  Capping of the top and right boundary components
  with discs yields the surface $S$ on the right.
  Now rotating about the $x$-axis by an angle of $\pi$
  yields an involution $\iota$ of $S$.
  Suppose the curves $\alpha_1, \dots, \alpha_n$
  are numbered such that $\alpha_n$ is the right-most cuve.
  Then $\iota$ preserves $\alpha_1, \dots, \alpha_{n-1}$,
  but not $\alpha_n$.
  Thus, the strict subgroup of $\cgroup(S)$ generated by $T_1, \dots, T_{n-1}$
  is isomorphic to $\artin(A_{n-1})$.
  This excludes the case $\Gamma = A_1$ immediately,
  and an application of Lemma~\ref{lem:almost-co-hopfian}
  excludes the case $\Gamma = A_{n-1}$.

  Next, we show that the inclusion-induced homomorphism
  $\cgroup(N^n \cup \Delta) \to \cgroup(S)$ is not injective.
  To this end, consider the boundary curve $\beta$
  of the chain $\alpha_1, \dots, \alpha_{n-1}$
  in $N^n \cup \Delta$.
  Then $\beta$ intersects $\alpha_n$ twice.
  But the image of $\beta$ under the inclusion map
  $N^n \cup \Delta \to S$ does not intersect
  the image of $\alpha_n$.
  Hence, the commutator $T_\beta T_n T_{\beta}^{-1} T_n^{-1}$
  is a non-trivial element of the kernel.
  By Lemma~\ref{lem:hopfian}, $\Gamma$ cannot be
  $D_n$, excluding all possibilities for $\Gamma$.
\end{proof}

\begin{figure}[htb]
  \centering
  \begin{minipage}{0.5\textwidth}
    \centering
    \includegraphics{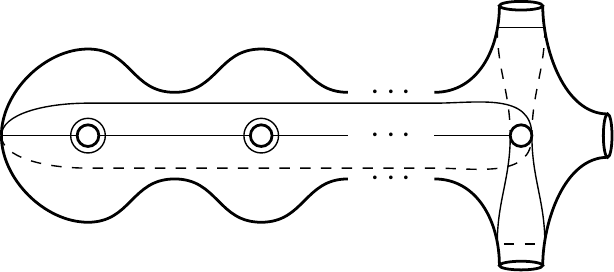}
  \end{minipage}%
  \begin{minipage}{0.5\textwidth}
    \centering
    \includegraphics{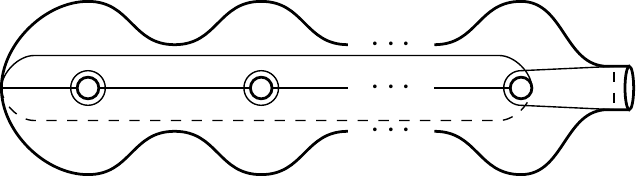}
  \end{minipage}%
  \caption{Another view of the surfaces $N^n$ and $N^n \cup \Delta^1 \cup \Delta^2$ for odd $n$.}%
  \label{fig:plastic-odd}
\end{figure}

\section{Proof of the main theorem}\label{sec:punctured-discs-and-annuli}
This short final section is about gluing in
punctured discs and annuli to the surfaces from Table~\ref{tab:summary}
and collecting the relevant results in this text to prove
Theorem~\ref{thm:cycle-relation}.

\begin{proposition}\label{prop:punctured-discs-and-annuli}
  Let $S$ be a surface containing
  a circuit $\alpha_1, \dots, \alpha_{n}$.
  Suppose that
  $\cgroup(S)$ is geometrically isomorphic
  to $\artin(\widetilde A_{n-1})$.
  Let $\Delta_1$ be a once-marked disc
  whose interior is disjoint from the interior of $S$,
  with $\partial \Delta_1 \subset \partial S$.
  Then the inclusion-induced
  homomorphism
  $\cgroup(S) \to \cgroup(S \cup \Delta_1)$
  is an isomorphism.
  Similarly,
  if $Z$ is an annulus
  whose interior is disjoint from the interior of $S$,
  with $\partial Z \subset \partial S$,
  then the inclusion-induced homomorphism
  $\cgroup(S) \to \cgroup(S \cup Z)$ is an isomorphism.
\end{proposition}

\begin{proof}
  Charney and Peifer show that
  for $n \geq 3$, the center of $\artin(\widetilde A_{n-1}$)
  is trivial~\cite[Proposition~1.3]{charney-peifer}.
  It now follows from Lemma~\ref{lem:inclusion-homo}
  that the inclusion-induced homomorphisms
  $\cgroup(S) \to \cgroup(S \cup \Delta_1)$
  and
  $\cgroup(S) \to \cgroup(S \cup Z)$
  are injective and hence isomorphisms.
\end{proof}

\begin{proof}[Proof of Theorem~\ref{thm:cycle-relation}]
  We prove the right-to-left implication,
  as the left-to-right implication follows from
  Labruère's result~\cite[Proposition~2]{labruere}.
  Contrapositively, suppose that
  the circuit $\alpha_1, \dots, \alpha_n$ does not bound
  an embedded closed disc.
  In other words, the complement of a regular neighbourhood
  of $\alpha_1, \dots, \alpha_n$ in $S$
  is a union of surfaces that are not embedded discs.
  Let $S'$ be the union of such a neighbourhood
  with all the non-embedded discs in its complement.
  Theorem~\ref{thm:regular-nbhd} and
  Proposition~\ref{prop:antilde} imply that $\cgroup(S')$
  is geometrically isomorphic to $\artin(\widetilde{A}_{n-1})$.
  The complement of $S'$ in $S$ is a union of
  surfaces that are not discs, so
  by Proposition~\ref{prop:punctured-discs-and-annuli}
  and Lemma~\ref{lem:inclusion-homo},
  it follows that also $\cgroup(S)$ is geometrically isomorphic
  to $\artin(\widetilde{A}_{n-1})$.
  But the cycle relation does not hold in this group.
  Indeed, as remarked above, the center of $\artin(\widetilde{A}_{n-1})$
  is trivial, whereas the quotient of $\artin(\widetilde{A}_{n-1})$
  by the normal subgroup generated by the cycle relation
  is isomorphic to $\artin(D_n)$ (see Lemma~\ref{lem:labruere} and Proposition~\ref{prop:dn}), which has infinite
  cyclic center~\cite[Satz~7.2]{brieskorn-saito}.
\end{proof}

\section*{Acknowledgements}
The author wishes to thank the anonymous referee for their thorough review
and the many small improvements that were suggested.

\bibliographystyle{amsalpha}
\bibliography{biblio}

\end{document}